\documentclass{amsart}
\usepackage{setspace}
\usepackage{a4}
\usepackage{amssymb,amsmath,amsthm,latexsym}
\usepackage{amsfonts}
\usepackage{amsfonts}
\usepackage{graphicx}
\usepackage{textcomp}
\usepackage{cite}
\usepackage{enumerate}
\usepackage[mathscr]{euscript}
\usepackage{mathtools}
\newtheorem{theorem}{Theorem}[section]

\newtheorem{conjecture}[theorem]{Conjecture}
\newtheorem{corollary}[theorem] {Corollary}
\newtheorem{definition}[theorem]{Definition}

\newtheorem{problem}[theorem]{Problem}
\newtheorem{proposition}[theorem]{Proposition}
\newtheorem{remark}[theorem]{Remark}

\title{This is the title}
\usepackage{amssymb}
\usepackage{amssymb}
\usepackage{amssymb}
\usepackage{amssymb}
\usepackage{amsmath}
\usepackage{tikz}
\usepackage{hyperref}
\usepackage{enumerate}
\usepackage{mathtools}
\usepackage{amsmath}
\usepackage{tikz}
\usepackage{amssymb}
\usepackage{amsmath}
\usepackage{tikz}
\usepackage{hyperref}
\usepackage{amssymb}
\usepackage{amssymb}
\usepackage{amssymb}
\usepackage{amssymb}
\usepackage{amsmath}
\usepackage{tikz}
\usepackage{hyperref}
\usepackage{enumerate}
\usepackage{mathtools}
\usepackage{amsmath}
\usepackage{tikz}
\usepackage{graphicx}
\usepackage{textcomp}
\usetikzlibrary{cd}
\usepackage{fancyhdr}

\begin{document}
\begin{center}
{\bf{MODULAR PAULSEN PROBLEM AND MODULAR PROJECTION PROBLEM}}\\
\vspace{.3cm}
\hrule
\vspace{.3cm}
{\bf{K. MAHESH KRISHNA}} \\
Post Doctoral Fellow \\
Statistics and Mathematics Unit\\
Indian Statistical Institute, Bangalore Centre\\
Karnataka 560 059 India\\
Email: kmaheshak@gmail.com \\

Date: \today
\end{center}

\hrule
\vspace{0.5cm}
\textbf{Abstract}:  Based on the solution of \textbf{Paulsen Problem} by Kwok, Lau, Lee,  and Ramachandran [\textit{STOC'18-Proceedings of the 50th Annual ACM SIGACT Symposium on Theory of Computing, 2018}] and independently by Hamilton, and Moitra [\textit{Isr. J. Math., 2021}] we study Paulsen Problem and Projection Problem in the context  of Hilbert C*-modules. 
We show that for commutative C*-algebras, if Modular Paulsen Problem has a solution, then Modular Projection Problem also has a solution. We formulate the problem of operator scaling for matrices over C*-algebras.


\hrule
\hrule
\section{Introduction}
We begin from the definition of frames and various special sub classes of it  for  finite dimensional Hilbert spaces which originated from the work of Duffin  and Schaeffer  \cite{DUFFIN, OLEBOOK}.
\begin{definition}\cite{HANKORNELSONLARSONWEBER, BENEDETTOFICKUS}
	A collection $\{\tau_j\}_{j=1}^n$ in  a finite dimensional Hilbert space $\mathcal{H}$ over $\mathbb{K} $ ($\mathbb{R} $ or $\mathbb{C} $) is said to be a \textbf{frame} for 	$\mathcal{H}$ if there exist  $a,b>0$ such that 
	\begin{align*}
		a\|h\|^2\leq \sum_{j=1}^{n}|\langle h,\tau_j \rangle |^2 \leq b \|h\|^2, \quad \forall h \in \mathcal{H}.
	\end{align*}
\end{definition}
\begin{definition}\cite{HANKORNELSONLARSONWEBER, BENEDETTOFICKUS}
	A frame  $\{\tau_j\}_{j=1}^n$ for a  finite dimensional Hilbert space $\mathcal{H}$  is said to be a 
	\begin{enumerate}[\upshape (i)]
		\item  \textbf{Parseval frame} if 
		\begin{align*}
			h= \sum_{j=1}^{n}\langle h,\tau_j \rangle \tau_j , \quad \forall h \in \mathcal{H}.
		\end{align*}
		\item \textbf{tight frame}  if there exists $a>0$ such that 
		\begin{align*}
			a\|h\|^2= \sum_{j=1}^{n}|\langle h,\tau_j \rangle |^2, \quad \forall h \in \mathcal{H}.
		\end{align*}
		\item \textbf{unit norm frame}   if $\|\tau_j\|=1$, for all $1\leq j\leq n$.
	\end{enumerate}
\end{definition}
From a  frame $\{\tau_j\}_{j=1}^n$ for $\mathcal{H}$ we get the following three operators.  We use $\mathcal{H}$ to denote a finite dimensional Hilbert space. We denote the identity operator on $\mathcal{H}$ by $I_\mathcal{H}.$
\begin{enumerate}[\upshape (i)]
	\item Analysis operator $\theta_\tau: \mathcal{H} \ni h \mapsto \theta_\tau h \coloneqq (\langle h, \tau_j \rangle ) _{j=1}^n\in \mathbb{K}^n$.
	\item Synthesis operator $\theta_\tau^*: \mathbb{K}^n \ni (a_j)_{j=1}^n \mapsto \theta_\tau^*(a_j)_{j=1}^n\coloneqq\sum_{j=1}^{n}a_j\tau_j \in \mathcal{H}$.
	\item Frame operator $S_\tau: \mathcal{H} \ni h \mapsto S_\tau h \coloneqq \sum_{j=1}^{n}\langle h,\tau_j \rangle \tau_j \in \mathcal{H}$.
\end{enumerate}
It then follows that  $S_\tau=\theta_\tau^*\theta_\tau$ and the  frame operator is positive and invertible. Therefore $\{S_{\tau}^{-1/2}\tau_j\}_{j=1}^n$  is a Parseval frame for $\mathcal{H}$. It is known that in this way  the construction of  Parseval frame  is difficult due to inversion and square root process of frame operator.  On the other side, given a collection,  it is easy (even not always) to verify that whether it is a frame than to verify that it is a Parseval frame. This leads to the following notions.  
\begin{definition}\cite{CAHILLCASAZZA} \label{ENF}
	A Parseval frame $\{\tau_j\}_{j=1}^n$ for a $d$-dimensional Hilbert space $\mathcal{H}$  is called an  \textbf{equal norm Parseval frame} if 
	\begin{align*}
		\|\tau_j\|^2=\frac{d}{n}, \quad \forall 1\leq j \leq n.
	\end{align*}
\end{definition}
\begin{definition}\cite{CAHILLCASAZZA}\label{NP}
	A  frame $\{\tau_j\}_{j=1}^n$ for a $d$-dimensional Hilbert space $\mathcal{H}$  is called an  \textbf{$\varepsilon$-nearly Parseval frame} ($\varepsilon<1$) if 	
	\begin{align*}
		(1-\varepsilon)I_\mathcal{H} \leq S_\tau \leq (1+\varepsilon)I_\mathcal{H}. 
	\end{align*}
\end{definition}
\begin{definition}\cite{CAHILLCASAZZA} \label{EENF}
	A  frame $\{\tau_j\}_{j=1}^n$ for a $d$-dimensional Hilbert space $\mathcal{H}$  is called an  \textbf{$\varepsilon$-nearly equal norm  frame} ($\varepsilon<1$) if 
	\begin{align*}
		(1-\varepsilon)\frac{d}{n}\leq 	\|\tau_j\|^2\leq (1+\varepsilon)\frac{d}{n}, \quad \forall 1\leq j \leq n.
	\end{align*}	
\end{definition}
\begin{definition}\cite{CAHILLCASAZZA}\label{BOTH}
	A frame for a $d$-dimensional Hilbert space $\mathcal{H}$ which is both $\varepsilon$-nearly  equal norm and  $\varepsilon$-nearly  Parseval is called as  \textbf{$\varepsilon$-nearly  equal norm Parseval frame.}
\end{definition}
Following notion is often used to measure the distance between two collections in Hilbert space.
\begin{definition}\cite{CAHILLCASAZZA}\label{DIS}
	\textbf{Distance} between  two collections $\{\tau_j\}_{j=1}^n$ and   $\{\omega_j\}_{j=1}^n$ in a Hilbert space $\mathcal{H}$ is defined as 
	\begin{align*}
		\operatorname{dist}(\{\tau_j\}_{j=1}^n, \{\omega_j\}_{j=1}^n)\coloneqq \left(\sum_{j=1}^{n}\|\tau_j-\omega_j\|^2\right)^\frac{1}{2}.
	\end{align*}
\end{definition}
Fundamental Paulsen problem now arises  from the work of Holmes and Paulsen \cite{HOLMESPAULSEN}. It is grounded on  following three results.
\begin{theorem}\cite{HOLMESPAULSEN}\label{PAULSENALGORITHM}
	There is an algorithm for turning every frame into an equal norm frame with same frame operator. 
\end{theorem}
\begin{theorem}\cite{CASAZZACUSTOM}\label{CLOSESTEQUALNORMFRAME}
	Let $\{\tau_j\}_{j=1}^n$ be an $\varepsilon$-nearly equal norm frame for  a $d$-dimensional Hilbert space $\mathcal{H}$. Then the closest equal norm frame to $\{\tau_j\}_{j=1}^n$ is given by $\{\omega_j\}_{j=1}^n$, where	
	\begin{align*}
		\omega_j\coloneqq \left(\frac{\sum_{k=1}^{n}\|\tau_k\|}{n}\right)\frac{\tau_j}{\|\tau_j\|}, \quad \forall 1\leq j \leq n.
	\end{align*}
\end{theorem}
\begin{theorem}\cite{BODMANNCASAZZA, CASAZZAWERTY}\label{CLOSESTPARSEVALFRAME}
	Let $\{\tau_j\}_{j=1}^n$ be a frame for  a $d$-dimensional Hilbert space $\mathcal{H}$. Then $\{S_\tau^{-1/2}\tau_j\}_{j=1}^n$ is the closest Parseval frame to $\{\tau_j\}_{j=1}^n$, i.e., 
	\begin{align*}
		\sum_{j=1}^{n}\left\|S_\tau^\frac{-1}{2}\tau_j-\tau_j\right\|^2=\inf\left\{ \sum_{j=1}^{n}\left\|\tau_j-\omega_j\right\|^2: \{\omega_j\}_{j=1}^n \text{ is a Parseval frame for }	\mathcal{H}\right \}.
	\end{align*}
	Further, $\{S_\tau^{-1/2}\tau_j\}_{j=1}^n$ is the unique minimizer. Moreover, if $\{\tau_j\}_{j=1}^n$   is any $\varepsilon$-nearly  Parseval frame for $\mathcal{H}$, then 
	\begin{align*}
		\sum_{j=1}^{n}\left\|S_\tau^\frac{-1}{2}\tau_j-\tau_j\right\|^2\leq d(2-\varepsilon-2\sqrt{1-\varepsilon})\leq \frac{d\varepsilon^2}{4}.
	\end{align*}
\end{theorem}
\begin{problem}\cite{CAHILLCASAZZA, CAHILLCASAZZAKUTYNIOK, CAHILLCASAZZAKUTYNIOK2} (\textbf{Paulsen problem})\label{PAUSENFIRST}
	\textbf{Find the function  $f: (0, 1)\times \mathbb{N} \times \mathbb{N} \to [0, \infty)$ so that for any 	$\varepsilon$-nearly  equal norm  Parseval frame $\{\tau_j\}_{j=1}^n$ for d-dimensional Hilbert space $\mathcal{H}$, there is an equal norm Parseval frame  $\{\omega_j\}_{j=1}^n$ for $\mathcal{H}$ satisfying 
		\begin{align*}
			\operatorname{dist}^2(\{\tau_j\}_{j=1}^n, \{\omega_j\}_{j=1}^n)= \sum_{j=1}^{n}\|\tau_j-\omega_j\|^2\leq f(\varepsilon, n,d).
		\end{align*}
		Moreover, whether $f$ depends on $n$?}
\end{problem}
It is clear that Problem  \ref{PAUSENFIRST} can also be stated as follows.
\begin{problem}\cite{CAHILLCASAZZA} (\textbf{Paulsen problem})\label{PAULSENSECOND}
	\textbf{Find the function $f: (0, 1)\times \mathbb{N} \times \mathbb{N} \to [0, \infty)$ so that for any  $\varepsilon$-nearly  equal norm Parseval frame $\{\tau_j\}_{j=1}^n$  for d-dimensional Hilbert space $\mathcal{H}$, 
		\begin{align*}
			\inf \{\operatorname{dist}^2(\{\tau_j\}_{j=1}^n, \{\omega_j\}_{j=1}^n):\{\omega_j\}_{j=1}^n \text{ is an equal norm Parseval frame for } \mathcal{H}\}\leq f(\varepsilon, n,d).
		\end{align*}
		Moreover, whether $f$ depends on $n$?}
\end{problem}
The function $f$ in Definition \ref{PAUSENFIRST} and in Definition \ref{PAULSENSECOND} is known as Paulsen function. First question is   whether such a function exists. Using compactness of the unit sphere in finite dimensional Hilbert space, Hadwin proved the following.
\begin{theorem}\cite{BODMANNCASAZZA}\label{BODMANNCASAZZAEXISTS}
	Solution $f$	to Paulsen problem exists.
\end{theorem}
Casazza \cite{CASAZZABOOK} observed  that Paulsen function is bounded below and it  is independent of the number of elements in the frame.
\begin{proposition}\cite{CASAZZABOOK}\label{CP}
	Paulsen function $f$ satisfies 
	\begin{align*}
		f(\varepsilon, n, d)\geq \varepsilon^2d, \quad \forall \varepsilon>0, \forall d \in \mathbb{N}.
	\end{align*}	
\end{proposition}
Using the system of ordinary differential equations and  frame energy, Bodmann and Casazza were able to obtain  the first striking result  for Paulsen problem in 2010 \cite{BODMANNCASAZZA}.
\begin{theorem}\cite{BODMANNCASAZZA}
	Let $n$ and $d$ be relatively prime and let $\varepsilon<1/2$. If $\{\tau_j\}_{j=1}^n$  is any $\varepsilon$-nearly equal norm frame  for a $d$-dimensional Hilbert space $\mathcal{H}$, then there is an equal norm Parseval frame  $\{\omega_j\}_{j=1}^n$ for $\mathcal{H}$ such that 
	\begin{align*}
		\operatorname{dist}^2(\{\tau_j\}_{j=1}^n, \{\omega_j\}_{j=1}^n)\leq \frac{29}{8}d^2n(n-1)^8\varepsilon.
	\end{align*}
\end{theorem}
Casazza, Fickus, and Mixon \cite{CASAZZAFICKUSMIXON} gave the following version of the Paulsen problem in 2012.
\begin{problem}\cite{CASAZZAFICKUSMIXON} (\textbf{Paulsen problem})\label{CASAZZASOLUTION}
	\textbf{Given $n,d \in \mathbb{N}$, find positive 	$\delta(n,d)$, $c(n,d)$ and $\alpha(n,d)$ such that given any unit norm frame $\{\tau_j\}_{j=1}^n$ for $d$-dimensional Hilbert space $\mathcal{H}$ satisfying 
		\begin{align*}
			\left\|S_\tau-\frac{n}{d}I_\mathcal{H}\right\|_{\operatorname{HS}}\leq \delta(n,d)	,
		\end{align*}
		there exists a unit norm tight frame $\{\omega_j\}_{j=1}^n$ for $\mathcal{H}$ such that 
		\begin{align*}
			\|\theta_\omega^*-\theta_\tau^*\|_{\operatorname{HS}} \leq c(n,d)	\left\|S_\tau-\frac{n}{d}I_\mathcal{H}\right\|_{\operatorname{HS}}^{\alpha(n,d)}	,
		\end{align*}
		where HS denotes the Hilbert-Schmidt norm of the operator}.
\end{problem}
In the same paper \cite{CASAZZAFICKUSMIXON}  Casazza, Fickus, and Mixon  proved the following  particular cases   of Problem  \ref{CASAZZASOLUTION}.
\begin{theorem}\cite{CASAZZAFICKUSMIXON}\label{PRIME}
	Let $n$ and $d$ be relatively prime and let $0<t<1/{2n}$.	Let $\{\tau^{(0)}_j\}_{j=1}^n$ be a unit norm tight frame for a $d$-dimensional Hilbert space $\mathcal{H}$ such that 
	\begin{align*}
		\left\|S^{(0)}_\tau-\frac{n}{d}I_\mathcal{H}\right\|_{\operatorname{HS}}^2\leq \frac{2}{d^3}	.
	\end{align*}
	Define $\{\omega_j^{(k)}\}_{j=1}^n$ as follows.
	\begin{align*}
		\omega_j^{(k)}\coloneqq S_\tau^{(k)}\tau_j^{(k)}-\langle S_\tau^{(k)}\tau_j^{(k)}, \tau_j^{(k)}\rangle \tau_j^{(k)}, \quad \forall 1 \leq j \leq n, \forall k \geq 0.	
	\end{align*}
	Now define $\{\tau_j^{(k)}\}_{j=1}^n$ as follows.
	\begin{align*}
		\tau_j^{(k+1)}\coloneqq 
		\begin{dcases}
			\cos(\|\omega_j^{(k)}\|t)\tau_j^{(k)}- \sin(\|\omega_j^{(k)}\|t)\frac{\omega_j^{(k)}}{\|\omega_j^{(k)}\|}& \text{if }  \omega_j^{(k)}\neq 0 \\
			\tau_j^{(k)} & \text{if }  \omega_j^{(k)}=0\\
		\end{dcases}
		,\quad \forall 1 \leq j \leq n, \forall k \geq 0.
	\end{align*}
	Then the limit of $\{\tau_j^{(k)}\}_{j=1}^n$ as $k\to \infty$ exists, denoted by $\{\tau_j^{(\infty)}\}_{j=1}^n$ is a unit norm tight frame for $\mathcal{H}$ and satisfies
	\begin{align*}
		\left\|S^{(\infty)}_\tau-S^{(0)}_\tau\right\|_{\operatorname{HS}}\leq \frac{4d^{20}n^{8.5}}{1-2nt}	\left\|S^{(0)}_\tau-\frac{n}{d}I_\mathcal{H}\right\|_{\operatorname{HS}}.
	\end{align*}
\end{theorem}
\begin{theorem}\cite{CASAZZAFICKUSMIXON}
	Let $\varepsilon\leq 1/{2n}$. If $\{\tau_j\}_{j=1}^n$  is any $\varepsilon$-orthogonally partitionable unit norm frame  for a $d$-dimensional Hilbert space $\mathcal{H}$, then there is an orthogonally partitionable  unit  norm  frame  $\{\omega_j\}_{j=1}^n$ for $\mathcal{H}$ such that 	
	\begin{align*}
		\|\theta_\omega^*-\theta_\tau^*\|_{\operatorname{HS}} \leq \sqrt{2n}	(\varepsilon d)^\frac{1}{3}.
	\end{align*}
\end{theorem}
\begin{theorem}\cite{CASAZZAFICKUSMIXON}
	Let $n$ and $d$ be not relatively prime. 	If $\{\tau_j\}_{j=1}^n$  is any  unit norm frame  for a $d$-dimensional Hilbert space $\mathcal{H}$, then there is a unit  norm  frame  $\{\omega_j\}_{j=1}^n$ for $\mathcal{H}$ which is either tight or is orthogonally partitionable with equal redundancies in each of the two partitioned subsets such that 
	\begin{align*}
		\|\theta_\omega^*-\theta_\tau^*\|_{\operatorname{HS}} \leq	3 d^\frac{6}{7}\sqrt{n}	\left\|S_\tau-\frac{n}{d}I_\mathcal{H}\right\|_{\operatorname{HS}}^\frac{1}{7}.
	\end{align*}
\end{theorem}
Projection problem is another prime  problem in finite dimensional Hilbert space theory.
\begin{problem}\cite{CAHILLCASAZZA}\label{PP} (\textbf{Projection problem})
	\textbf{	Let $\mathcal{H}$ be a d-dimensional Hilbert space with orthonormal basis $\{u_k\}_{k=1}^d$. Find the function $g: (0, 1)\times \mathbb{N} \times \mathbb{N} \to [0, \infty)$ satisfying the following: If $P: \mathcal{H} \to \mathcal{H}$ is an orthogonal projection  of rank $n$ satisfying 
		\begin{align*}
			(1-\varepsilon)\frac{n}{d}\leq 	\|Pu_k\|^2\leq (1+\varepsilon)\frac{n}{d}, \quad \forall 1\leq k \leq d,
		\end{align*}
		then there exists an orthogonal projection $Q: \mathcal{H} \to \mathcal{H}$ with 
		\begin{align*}
			\|Qu_k\|^2=\frac{n}{d},	\quad \forall 1\leq k \leq d,
		\end{align*}
		satisfying 
		\begin{align*}
			\sum_{k=1}^{d}\|Pu_k-Qu_k\|^2\leq g(\varepsilon, n,d).	
	\end{align*}
	Moreover, whether $g$ depends on $n$?}
\end{problem}
By employing the chordal distance between subspaces \cite{CONWAYHARDIN}, Cahill and Casazza \cite{CAHILLCASAZZA}  showed that Paulsen problem can be solved  if and only if  projection problem can be solved.
\begin{theorem}\cite{CAHILLCASAZZA}
	If $f$ is the function for the Paulsen problem and $g$ is the function for the projection problem, then 
	\begin{align*}
		g(\varepsilon, n,d)\leq 4 f(\varepsilon, n,d)\leq 8g(\varepsilon, n,d)		, \quad \forall \varepsilon, n, d.
	\end{align*}
\end{theorem}
Again, in the same paper,   Cahill and Casazza \cite{CAHILLCASAZZA} were able to derive the following using    Naimark complement    of frames (see \cite{CASAZZAFICKUS2, CASAZZAKUTYNIOK, CASAZZALYNCH, CZAJA}).   
\begin{theorem}\cite{CAHILLCASAZZA}\label{COMPLEMENTIMPLIES}
	Let $n>d$. If $f$ is the function for the Paulsen problem, then 
	\begin{align*}
		f(\varepsilon, n,d) \leq 8 f\left(\frac{d}{n-d}, n, n-d\right).
	\end{align*}
\end{theorem}
A big benefit of Theorem \ref{COMPLEMENTIMPLIES} is the following result. 
\begin{theorem}\cite{CAHILLCASAZZA}\label{LESSTHAN}
	To solve the Paulsen problem for a $d$-dimensional Hilbert space $\mathcal{H}$, it suffices to solve it for  Parseval frames $\{\tau_j\}_{j=1}^n$ for $\mathcal{H}$ with $d\leq n\leq 2d$.
\end{theorem}
In 2017, using operator scaling algorithm and smoothed analysis, Kwok, Lau, Lee, and Ramachandran \cite{KWOKLAULEERAMACHANDRAN, KWOKLAULEERAMACHANDRAN1} resolved Paulsen problem by deriving the following result.
\begin{theorem}\cite{KWOKLAULEERAMACHANDRAN, KWOKLAULEERAMACHANDRAN1, RAMACHANDRANTHESIS, RAMACHANDRANOLIVEIRA}
	\label{RAMA}  \textbf{(Kwok-Lau-Lee-Ramachandran Theorem)	For any  $\varepsilon$-nearly  equal norm Parseval frame $\{\tau_j\}_{j=1}^n$  for $\mathbb{R}^d$, there 	is an equal norm Parseval frame  $\{\omega_j\}_{j=1}^n$ for $\mathbb{R}^d$ satisfying 
	\begin{align*}
		\operatorname{dist}^2(\{\tau_j\}_{j=1}^n, \{\omega_j\}_{j=1}^n)\leq O(\varepsilon d^\frac{13}{2}).
	\end{align*}
	In other words, $f$ does not depend upon $n$ and 	 $f(\varepsilon, n,d)=c\varepsilon d^\frac{13}{2}$, for some constant $c>0$.}
\end{theorem}
In 2018, using radial isotropic position, Hamilton and Moitra \cite{HAMILTONMOITRA, HAMILTONMOITRA2} gave another proof of Paulsen problem which improved Theorem \ref{RAMA}.
\begin{theorem}\cite{HAMILTONMOITRA, HAMILTONMOITRA2} 
 \textbf{(Hamilton-Moitra Theorem)	For any  $\varepsilon$-nearly  equal norm Parseval frame $\{\tau_j\}_{j=1}^n$  for $\mathbb{R}^d$, there 	is an equal norm Parseval frame  $\{\omega_j\}_{j=1}^n$ for $\mathbb{R}^d$ satisfying 
	\begin{align*}
		\operatorname{dist}^2(\{\tau_j\}_{j=1}^n, \{\omega_j\}_{j=1}^n)\leq 20 \varepsilon d^2.
	\end{align*}
	In other words,  $f(\varepsilon, n,d)=20\varepsilon d^2$.}
\end{theorem}
After the solution of Paulsen problem (and hence projection problem), both Paulsen and projection problems are stated for Banach spaces which is open till today \cite{MAHESHKRISHNA2}. In this paper, we formulate Paulsen and projection problems for Hilbert C*-modules based on frame theory for Hilbert C*-modules. We show that solution of modular Paulsen problem gives a solution to modular projection problem.

\section{Modular Paulsen Problem and Modular Projection Problem}
Originated from the work of Kaplansky \cite{KAPLANSKY} for commutative C*-algebras and developed from the work of Paschke  \cite{PASCHKE} and Rieffel \cite{RIEFFEL}   for non commutative C*-algebras, through a development of half century, it became evident that Hilbert C*-modules play a prominent role in the non commutative geometry \cite{LANCE, WEGGEOLSEN, MANUILOVTROITSKY}. This demanded a necessity of developing frame theory for Hilbert C*-modules.   In their fundamental paper Frank and Larson introduced the notion of frames for Hilbert C*-modules \cite{FRANKLARSON}. In this paper we only consider the following particular Hilbert C*-module. Let $\mathcal{A}$ be a  unital C*-algebra, $d \in \mathbb{N}$  and $\mathcal{A}^d$ be the left module over $\mathcal{A}$ w.r.t. natural operations. Modular $\mathcal{A}$-inner product on $\mathcal{A}^d$ is defined as 
\begin{align*}
	\langle (x_j)_{j=1}^d, (y_j)_{j=1}^d\rangle \coloneqq \sum_{j=1}^{d}x_jy_j^*,\quad \forall (x_j)_{j=1}^d, (y_j)_{j=1}^d \in \mathcal{A}^d.
\end{align*}
Hence the norm on $\mathcal{A}^d$ becomes 
\begin{align*}
	\|(x_j)_{j=1}^d\|\coloneqq \left\|\sum_{j=1}^{d}x_jx_j^*\right\|^\frac{1}{2}, \quad \forall (x_j)_{j=1}^d \in \mathcal{A}^d.
\end{align*}
Standard orthonormal basis for $\mathcal{A}^d$ is denoted by $\{e_j\}_{j=1}^d$. 
\begin{definition}\cite{FRANKLARSON}
A collection $\{\tau_j\}_{j=1}^n$ in  $\mathcal{A}^d$ is said to be a  (modular) \textbf{frame} for 	$\mathcal{A}^d$ if there exist  real $a,b>0$ such that 
\begin{align*}
	a\langle x,x \rangle\leq \sum_{j=1}^{n}\langle x,\tau_j \rangle \langle \tau_j, x \rangle\leq b \langle x,x \rangle, \quad \forall x \in \mathcal{A}^d.
\end{align*}
A frame $\{\tau_j\}_{j=1}^n$ for   $\mathcal{A}^d$ is said to be  \textbf{Parseval} if 
\begin{align*}
	\langle x,x \rangle= \sum_{j=1}^{n}\langle x,\tau_j \rangle \langle \tau_j, x \rangle , \quad \forall x \in \mathcal{A}^d.
\end{align*}	
\end{definition}
It has to be noted that the theory of frames for Hilbert C*-modules behaves mysteriously compared to the theory of frames for Hilbert spaces \cite{ARAMBASIC, JING, RAEBURN, ASADIFRANK, LI, ASADIFRANK2}. Thus every result of Hilbert space frame theory has to be rechecked for Hilbert C*-modules. Through detailed working, Frank and Larson obtained the following result. 
\begin{theorem}\cite{FRANKLARSON}
Let $\{\tau_j\}_{j=1}^n$ be a frame for   $\mathcal{A}^d$. Then 
\begin{enumerate}[\upshape (i)]
	\item Analysis homomorphism  $\theta_\tau: \mathcal{A}^d \ni x \mapsto \theta_\tau x \coloneqq (\langle x, \tau_j \rangle ) _{j=1}^n\in \mathcal{A}^n$ is adjointable and bounded below. 
	\item Synthesis  homomorphism $\theta_\tau^*: \mathcal{A}^n \ni (a_j)_{j=1}^n \mapsto \theta_\tau^*(a_j)_{j=1}^n\coloneqq\sum_{j=1}^{n}a_j\tau_j \in \mathcal{A}^d$  is surjective
	\item Frame homomorphism $S_\tau: \mathcal{A}^d \ni x \mapsto S_\tau x \coloneqq \sum_{j=1}^{n}\langle x,\tau_j \rangle \tau_j \in \mathcal{A}^d$ is self-adjoint positive and invertible.
	\item $S_\tau=\theta_\tau^*\theta_\tau$.
	\item $P_\tau \coloneqq \theta_\tau S_\tau^{-1}\theta_\tau^*$ is a projection onto $\theta_\tau(\mathcal{A}^d)$. 
\end{enumerate}
\end{theorem}
To state modular Paulsen problem  we  need   modular  versions of Definitions \ref{ENF}, \ref{NP}, \ref{EENF}  \ref{BOTH} and \ref{DIS} (see \cite{MAHESHKRISHNA}).
\begin{definition}
	A Parseval frame $\{\tau_j\}_{j=1}^n$ for $\mathcal{A}^d$  is called an  \textbf{equal inner product Parseval frame} if 
	\begin{align*}
	\langle \tau_j, \tau_j \rangle =\frac{d}{n}, \quad \forall 1\leq j \leq n.
	\end{align*}
\end{definition}
\begin{definition}
	A  frame $\{\tau_j\}_{j=1}^n$ for $\mathcal{A}^d$  is called an  \textbf{$\varepsilon$-nearly Parseval frame} if 	
	\begin{align*}
		(1-\varepsilon)I_{\mathcal{A}^d} \leq S_\tau \leq (1+\varepsilon)I_{\mathcal{A}^d}.
	\end{align*}
\end{definition}
\begin{definition}
	A  frame $\{\tau_j\}_{j=1}^n$ for $\mathcal{A}^d$  is called an  \textbf{$\varepsilon$-nearly equal inner product  frame} if 
	\begin{align*}
		(1-\varepsilon)\frac{d}{n}\leq \langle \tau_j, \tau_j \rangle\leq (1+\varepsilon)\frac{d}{n}, \quad \forall 1\leq j \leq n.
	\end{align*}	
\end{definition}
\begin{definition}
		A frame for  $\mathcal{A}^d$ which is both $\varepsilon$-nearly  equal inner product and  $\varepsilon$-nearly  Parseval is called as  \textbf{$\varepsilon$-nearly  equal inner product  Parseval frame.}
\end{definition}
\begin{definition}\label{MD}
	\textbf{Modular distance} between  two collections $\{\tau_j\}_{j=1}^n$,  $\{\omega_j\}_{j=1}^n$ in  $\mathcal{A}^d$ is defined as 
	\begin{align*}
		\operatorname{dist}(\{\tau_j\}_{j=1}^n, \{\omega_j\}_{j=1}^n)\coloneqq \left\|\sum_{j=1}^{n}\langle \tau_j-\omega_j, \tau_j-\omega_j\rangle \right\|^\frac{1}{2}.
	\end{align*}
\end{definition}
In view of Theorems \ref{PAULSENALGORITHM},   \ref{CLOSESTEQUALNORMFRAME}   and \ref{CLOSESTPARSEVALFRAME} we ask following three problems.
\begin{problem}
	\textbf{Whether 	there is an algorithm for turning every modular frame  into an equal inner product modular frame with same modular frame homomorphism?}
\end{problem}
\begin{problem}
	\textbf{What is the closest (in terms of distance given in Definition \ref{MD}) modular Parseval frame  to  a given modular frame?}
\end{problem}
\begin{problem}
	\textbf{What is the closest equal inner product modular frame   (in terms of distance given in Definition \ref{MD}) to a given $\varepsilon$-nearly  inner product modular frame?}
\end{problem}
We can now formulate the most important problem of this paper.
\begin{problem} (\textbf{Modular Paulsen problem)\label{MOD123}
Find the function  $f: (0, \infty)\times \mathbb{N} \times \mathbb{N} \to [0, \infty)$ so that for any 	$\varepsilon$-nearly  equal inner product   Parseval frame $\{\tau_j\}_{j=1}^n$ for  $\mathcal{A}^d$, there is an equal inner product  Parseval frame  $\{\omega_j\}_{j=1}^n$ for $\mathcal{A}^d$ satisfying 
\begin{align*}
	\operatorname{dist}^2(\{\tau_j\}_{j=1}^n, \{\omega_j\}_{j=1}^n)= \left\|\sum_{j=1}^{n}\langle \tau_j-\omega_j, \tau_j-\omega_j\rangle \right\|\leq f(\varepsilon, n,d).
\end{align*}
Moreover, whether $f$ depends on $n$?}	
\end{problem}
It is clear that, just like Hilbert spaces,  we can reformulate Problem \ref{MOD123}  as follows.
\begin{problem} (\textbf{Modular Paulsen problem)
	Find the function $f: (0, \infty)\times \mathbb{N} \times \mathbb{N} \to [0, \infty)$ so that for any  $\varepsilon$-nearly  equal inner product  Parseval frame $\{\tau_j\}_{j=1}^n$    for  $\mathcal{A}^d$,
	\begin{align*}
		\inf \{\operatorname{dist}^2(\{\tau_j\}_{j=1}^n, \{\omega_j\}_{j=1}^n):\{\omega_j\}_{j=1}^n \text{ is an equal inner product  Parseval frame for } \mathcal{A}^d\}\leq f(\varepsilon, n,d).
	\end{align*}
	Moreover, whether $f$ depends on $n$?}
\end{problem}
Theorem \ref{BODMANNCASAZZAEXISTS} says that solution of Paulsen problem exists for Hilbert spaces. However, the proof  of Theorem \ref{BODMANNCASAZZAEXISTS}  as given by Hadwin can not be executed for C*-algebras. Therefore, we first ask the following question.
\begin{problem}
	Whether there exists a solution to the modular Paulsen problem?
\end{problem}

Since Hilbert C*-modules are generalizations of Hilbert spaces,  we must have the following result (see Proposition \ref{CP}).
\begin{proposition}
Modular 	Paulsen function $f$,  if it exists,  satisfies 
	\begin{align*}
		f(\varepsilon, n, d)\geq \varepsilon^2d, \quad \forall \varepsilon>0, \forall d \in \mathbb{N}.
	\end{align*}
\end{proposition}
We formulate Problem \ref{PP} to Hilbert C*-modules as follows. 
\begin{problem}(\textbf{Modular projection problem)
Let $\mathcal{A}$ be a C*-algebra with  invariant basis number property.   Let $\{e_k\}_{k=1}^d$ be the standard   orthonormal basis for $\mathcal{A}^d$. Find the function $g: (0, \infty)\times \mathbb{N} \times \mathbb{N} \to [0, \infty)$ satisfying the following: If $P: \mathcal{A}^d \to \mathcal{A}^d$ is an orthogonal projection  of rank $n$ satisfying 
\begin{align*}
	(1-\varepsilon)\frac{n}{d}\leq 	\langle Pe_k, Pe_k \rangle \leq (1+\varepsilon)\frac{n}{d}, \quad \forall 1\leq k \leq d,
\end{align*}
then there exists an orthogonal projection $Q: \mathcal{A}^d \to \mathcal{A}^d$ with 
\begin{align*}
	\langle Qe_k, Qe_k\rangle =\frac{n}{d},	\quad \forall 1\leq k \leq d,
\end{align*}
satisfying 
\begin{align*}
\left\|	\sum_{k=1}^{d}\langle Pe_k-Qe_k, Pe_k-Qe_k\rangle \right \|\leq g(\varepsilon, n,d).	
\end{align*}
	Moreover, whether $g$ depends on $n$?}	
\end{problem}
First natural question is the following.
\begin{problem}\label{MODULARPAUSENPROJECTIONQUESTION}
	\textbf{Whether there is a relation between modular Paulsen problem   and modular projection problem?}
\end{problem}
We answer Problem \ref{MODULARPAUSENPROJECTIONQUESTION} partially. More precisely,  we show that if we can solve modular Paulsen problem, then we can solve modular projection problem. Our developments are motivated from the arguments given in \cite{CAHILLCASAZZA}.  First we derive a lemma.
\begin{theorem}\label{IMP}
Assume that the C*-algebra $\mathcal{A}$ is commutative. Let $\{\tau_j\}_{j=1}^n$ and $\{\omega_j\}_{j=1}^n$ be two Parseval frames for $\mathcal{A}^d$. If 	
\begin{align*}
	\operatorname{dist}^2(\{\tau_j\}_{j=1}^n, \{\omega_j\}_{j=1}^n)= \left\|\sum_{j=1}^{n}\langle \tau_j-\omega_j, \tau_j-\omega_j\rangle \right\|<\varepsilon, 
\end{align*}
then 
\begin{align*}
	\operatorname{dist}^2(\{\theta_\tau\tau_j\}_{j=1}^n, \{\theta_\omega\omega_j\}_{j=1}^n)= \left\|\sum_{j=1}^{n}\langle \theta_\tau\tau_j-\theta_\omega\omega_j, \theta_\tau\tau_j-\theta_\omega \omega_j\rangle \right\|<4\varepsilon.
\end{align*}
\end{theorem}
  \begin{proof}
  	 Let $ 1\leq j \leq n$. Then 
  	\begin{align*}
  \theta_\tau\tau_j=\sum_{k=1}^{n} \langle \tau_j , \tau_k \rangle e_k, \quad \theta_\omega\omega_j=\sum_{k=1}^{n} \langle \omega_j , \omega_k \rangle e_k.
  	\end{align*}
 We now note the following. 	If $a, b$  are any two elements of a C*-algebra, then $(a+b)(a+b)^*\leq 2(aa^*+bb^*).$ In fact, 
 	\begin{align*}
 	2(aa^*+bb^*)-(a+b)(a+b)^*&=2aa^*+2bb^*-aa^*-ab^*-ba^*-bb^*\\
 	&=aa^*+bb^*-ab^*-ba^*=a(a^*-b^*)-b(a^*-b^*)\\
 	&=(a-b)(a-b)^*\geq 0.
 \end{align*}
 Therefore 
  \begin{align*}
 \langle \theta_\tau\tau_j-\theta_\omega\omega_j, \theta_\tau\tau_j-\theta_\omega \omega_j\rangle&= \sum_{k=1}^n (\langle \tau_j , \tau_k \rangle-\langle \omega_j , \omega_k \rangle)	(\langle \tau_j , \tau_k \rangle-\langle \omega_j , \omega_k \rangle)^*\\
 &=\sum_{k=1}^n (\langle \tau_j , \tau_k -\omega_k \rangle+\langle \tau_j-\omega_j , \omega_k \rangle)	(\langle \tau_j , \tau_k -\omega_k \rangle+\langle \tau_j-\omega_j , \omega_k \rangle)^*\\
 &\leq 2 \sum_{k=1}^n \langle \tau_j , \tau_k -\omega_k \rangle  \langle \tau_j , \tau_k -\omega_k \rangle ^*+2 \sum_{k=1}^n\langle  \tau_ j-\omega_j , \omega_k\rangle\langle  \tau_j -\omega_j, \omega_k\rangle^*\\
 &=2 \sum_{k=1}^n \langle \tau_j , \tau_k -\omega_k \rangle  \langle \tau_j , \tau_k -\omega_k \rangle ^*+2\langle  \tau_ j-\omega_j ,  \tau_ j-\omega_j \rangle.
  \end{align*}
We now sum by varying $j$ and use commutativity of the C*-algebra to get
\begin{align*}
	\sum_{j=1}^{n}\langle \theta_\tau\tau_j-\theta_\omega\omega_j, \theta_\tau\tau_j-\theta_\omega\omega_j\rangle &\leq2 	\sum_{j=1}^{n}\sum_{k=1}^n \langle \tau_j , \tau_k -\omega_k \rangle  \langle \tau_j , \tau_k -\omega_k \rangle ^*+2	\sum_{j=1}^{n} \langle  \tau_ j-\omega_j ,  \tau_ j-\omega_j \rangle\\
	&=2\sum_{k=1}^n\sum_{j=1}^ n\langle \tau_j , \tau_k -\omega_k \rangle  \langle \tau_j , \tau_k -\omega_k \rangle ^*+2	\sum_{j=1}^{n} \langle  \tau_ j-\omega_j ,  \tau_ j-\omega_j \rangle\\
	&=2\sum_{k=1}^n\langle  \tau_ k-\omega_k ,  \tau_ k-\omega_k \rangle+2	\sum_{j=1}^{n} \langle  \tau_ j-\omega_j ,  \tau_ j-\omega_j \rangle\\
	&=4\sum_{j=1}^{n} \langle  \tau_ j-\omega_j ,  \tau_ j-\omega_j \rangle.
\end{align*}
By taking norm we get the conclusion.
  \end{proof}
\begin{theorem}
Assume that the C*-algebra is commutative.	If Modular Paulsen Problem has a solution, then Modular Projection Problem also has a solution. 
\end{theorem}
\begin{proof}
Let $\mathcal{A}$ be a commutative C*-algebra, $d\in \mathbb{N}$.   Let $\{e_k\}_{k=1}^d$ be the standard   orthonormal basis for $\mathcal{A}^d$. Assume that Modular Paulsen Problem has a solution. We wish to show that Modular Projection Problem also has a solution.   Let $P: \mathcal{A}^d \to \mathcal{A}^d$ be an orthogonal projection  of rank $n$ satisfying 
\begin{align}\label{P}
	(1-\varepsilon)\frac{n}{d}\leq 	\langle Pe_k, Pe_k \rangle \leq (1+\varepsilon)\frac{n}{d}, \quad \forall 1\leq k \leq d. 
\end{align}
Since $P$ is a projection, it is then clear that $\{Pe_k\}_{k=1}^d$ is a modular Parseval frame for  $P(\mathcal{A}^d)$. Note that Inequality (\ref{P}) tells that $\{Pe_k\}_{k=1}^d$ is an  $\varepsilon$-nearly  equal inner product   Parseval frame  for  $P(\mathcal{A}^d)$ (note that rank of this module is $n$). Since Modular Paulsen Problem has a solution, there is an equal inner product  Parseval frame  $\{\omega_k\}_{k=1}^d$ for $P(\mathcal{A}^d)$ satisfying 
\begin{align*}
	\operatorname{dist}^2(\{Pe_k\}_{k=1}^d, \{\omega_k\}_{k=1}^d)= \left\|\sum_{k=1}^{d}\langle Pe_k-\omega_k, Pe_k-\omega_k\rangle \right\|\leq f(\varepsilon, d, n).
\end{align*}
Define $Q\coloneqq P_\omega=\theta_\omega \theta_\omega^*:\mathcal{A}^d \to \mathcal{A}^d$. Then $Q$ is an orthogonal projection and $Qe_k=\theta_\omega \theta_\omega^*e_k=\theta_\omega \omega_k$ for all $1\leq k \leq d$. Since $\{\omega_k\}_{k=1}^d$ is an equal inner product Parseval frame for $P(\mathcal{A}^d)$, we have 
\begin{align*}
	\langle Qe_k, Qe_k \rangle=\langle \theta_\omega \omega_k, \theta_\omega\omega_k \rangle=\langle \omega_k, \omega_k \rangle =\frac{n}{d}, \quad \forall 1\leq k\leq d.
\end{align*} 
Define $\tau_k\coloneqq Pe_k$ for all $1\leq k \leq d$. Then we find that 
\begin{align*}
\theta_\tau Pe_k=\sum_{j=1}^{d}\langle Pe_k, \tau_j \rangle e_j=\sum_{j=1}^{d}\langle Pe_k, Pe_j \rangle e_j=\sum_{j=1}^{d}\langle Pe_k, e_j \rangle e_j=Pe_k, \quad \forall 1\leq k \leq d.
\end{align*}
Now using  Theorem \ref{IMP} we get
\begin{align*}
	\left\|	\sum_{k=1}^{d}\langle Pe_k-Qe_k, Pe_k-Qe_k\rangle \right \|&= \left\|\sum_{k=1}^{d}\langle Pe_k-\omega_k, Pe_k-\omega_k\rangle \right\|\\
	&=\left\|\sum_{k=1}^{d}\langle \theta_\tau Pe_k-\theta_\omega e_k, \theta_\tau P e_k-\theta_\omega e_k\rangle \right\|\\
	&\leq 4f(\varepsilon, d, n).
\end{align*}
Therefore Modular Projection Problem holds.
\end{proof}
\begin{corollary}
Assume that the C*-algebra is commutative.	If Modular Paulsen function $f$ exists, then modular projection function $g$ also exists and 
\begin{align*}
	g(\varepsilon, n,d)\leq 4 f(\varepsilon, n,d), \quad \forall (\varepsilon, n,d)\in (0, \infty)\times \mathbb{N} \times \mathbb{N}. 
\end{align*}	
\end{corollary}
\section{Appendix}
As mentioned in the introduction, Paulsen problem is solved in \cite{KWOKLAULEERAMACHANDRAN, KWOKLAULEERAMACHANDRAN1}  using operator scaling methods \cite{KWOKLAULEERAMACHANDRAN, GURVITS, KWOKLAULEERAMACHANDRAN1, GARGGURVITSOLIVEIRAWIGDERSON, GARGGURVITSOLIVEIRAWIGDERSON2, KWOKLAURAMACHANDRAN}. Therefore we wish to set following problems which are important. For this, we need some notions. 
Let  $\mathcal{A}$ be a unital C*-algebra and  $m, n$ be a natural numbers. Define  $M_{m\times n}(\mathcal{A})$  as the set of all $m$ by $n$ matrices over $\mathcal{A}$  with  natural matrix operations. The adjoint  of an element  $A\coloneqq [a_{j,k}]_{1\leq j\leq m, 1\leq k \leq n}\in M_{m\times n}(\mathcal{A})$ is defined as $A^*\coloneqq [a_{k,j}^*]_{1\leq k \leq n, 1\leq j \leq m}\in M_{n\times m}(\mathcal{A})$. We define the 	\textbf{modular Hilbert-Schmidt inner product}  of $A\coloneqq [a_{j,k}]_{1\leq j\leq m, 1\leq k \leq n}, B\coloneqq [b_{j,k}]_{1\leq j\leq m, 1\leq k \leq n}\in M_{m\times n}(\mathcal{A})$ as 
	\begin{align*}
	\langle A, B \rangle _{\operatorname{MHS}}\coloneqq \sum_{j=1}^m  \sum_{k=1}^n a_{j,k}b_{j,k}^*.
	\end{align*}
We define  	\textbf{modular Hilbert-Schmidt norm} of $A\coloneqq [a_{j,k}]_{1\leq j\leq m, 1\leq k \leq n}\in M_{m\times n}(\mathcal{A})$ as 
\begin{align*}
	\|A\|_{\operatorname{MHS}}\coloneqq \|	\langle A, A \rangle _{\operatorname{MHS}}\|^\frac{1}{2}=\left\|\sum_{j=1}^m  \sum_{k=1}^n a_{j,k}a_{j,k}^*\right\|^\frac{1}{2}.
\end{align*}
Whenever $n=m$, we denote $M_{m\times n}(\mathcal{A})$ by $M_{m}(\mathcal{A})$. Identity matrix in $M_{m}(\mathcal{A})$ is denoted by $I_m$. 
\begin{problem}
\textbf{(Modular Morphism Scaling  Problem)	Let $m,n \in \mathbb{N}$ and $\mathcal{A}$ be a unital C*-algebra. Given $U_1, \dots, U_k  \in M_{m\times n}(\mathcal{A})$, find matrices $L\in M_{m}(\mathcal{A})$ and  $R\in M_{n}(\mathcal{A})$ such that  matrices  
\begin{align*}
	V_j\coloneqq LV_j R, \quad \forall 1\leq j \leq k 
\end{align*}
satisfy  
\begin{align*}
	\sum_{j=1}^{k}V_jV_j^*=I_m \quad \text{and} \quad 	\sum_{j=1}^{k}V_j^*V_j=\frac{m}{n}I_n.	
\end{align*}
Moreover, whether there is a polynomial time algorithm for modular morphism scaling problem?}
\end{problem}
\begin{definition}
	(\textbf{Doubly balanced modular matrix  tuples}) Let $\mathcal{A}$ be a unital C*-algebra.   A modular matrix tuple $(V_1, \dots,  V_k)$, $V_j \in M_{m\times n}(\mathcal{A})$, $ \forall 1\leq j \leq k$ is called 	\textbf{doubly balanced} if there exists a positive element $c\in \mathcal{A}$ such that 
	\begin{align*}
		\sum_{j=1}^{k}V_jV_j^*=cnI_m \quad \text{and} \quad 	\sum_{j=1}^{k}V_j^*V_j=cmI_n.
	\end{align*}
\end{definition}
\begin{definition}
		(\textbf{Doubly stochastic  modular matrix  tuples}) Let $\mathcal{A}$ be a unital C*-algebra. A modular matrix tuple $(V_1, \dots, V_k)$, $V_j \in M_{m\times n}(\mathcal{A})$, $ \forall 1\leq j \leq k$ is called 	\textbf{doubly stochastic}  if 
	\begin{align*}
		\sum_{j=1}^{k}V_jV_j^*=I_m \quad \text{and} \quad 	\sum_{j=1}^{k}V_j^*V_j=\frac{1}{n}mI_n.
	\end{align*}
\end{definition}
\begin{definition}
A modular matrix tuple $(V_1, \dots, V_k)$, $V_j \in M_{m\times n}(\mathcal{A})$, $ \forall 1\leq j \leq k$ is called 	\textbf{$\varepsilon$-nearly doubly stochastic} if 	
\begin{align*}
(1-\varepsilon)I_m\leq 	\sum_{j=1}^{k}V_jV_j^* \leq (1+\varepsilon)I_m
\end{align*}
and 
\begin{align*}
		(1-\varepsilon)\frac{m}{n}I_n\leq \sum_{j=1}^{k}V_j^*V_j\leq (1+\varepsilon)\frac{m}{n}I_n.
\end{align*}
\end{definition}
\begin{definition}
Let $\mathcal{A}$ be a unital C*-algebra. 	Given modular matrix tuples $(U_1, \dots, U_k)$ and  $(V_1, \dots, V_k)$, where  $U_j, V_j \in M_{m\times n}(\mathcal{A})$, $ \forall 1\leq j \leq k$, we define the \textbf{modular distance} between $(U_1, \dots, U_k)$ and  $(V_1, \dots, V_k)$ as 
	\begin{align*}
		\operatorname{dist}((U_j)_{j=1}^k, \{V_j\}_{j=1}^k)\coloneqq \left\|\sum_{j=1}^{k} 	\langle U_j-V_j, U_j-V_j \rangle _{\operatorname{MHS}}\right\|^\frac{1}{2}.
	\end{align*}	
\end{definition}
\begin{problem}
\textbf{(Matrix Modular Paulsen Problem) Let $\mathcal{A}$ be a unital C*-algebra.  Find the function  $h: (0, \infty)\times \mathbb{N} \times \mathbb{N} \times \mathbb{N} \to [0, \infty)$ so that for any 	modular $\varepsilon$-nearly    doubly stochastic matrix tuple $(U_1, \dots, U_k)$, $U_j\in M_{m\times n}(\mathcal{A})$, $ \forall 1\leq j \leq k$,  there is a modular doubly stochastic matrix tuple   $(V_1, \dots, V_k)$, $V_j\in M_{m\times n}(\mathcal{A})$, $ \forall 1\leq j \leq k$ satisfying 
\begin{align*}
	\operatorname{dist}^2((U_j)_{j=1}^k, \{V_j\}_{j=1}^k)\coloneqq \left\|\sum_{j=1}^{k} 	\langle U_j-V_j, U_j-V_j \rangle _{\operatorname{MHS}}\right\| \leq h(\varepsilon, n,d, k).
\end{align*}
Moreover, whether $f$ depends on $n$?}	
\end{problem}
Next we wish to set up problem based on Barthe theorem \cite{BARTHE} which played important role in the proof of Hamilton and Moitra \cite{HAMILTONMOITRA, HAMILTONMOITRA2} for Paulsen Problem. To formulate we set the definition of radial isotropic position  vectors in modules as follows. 
\begin{definition}
Let $\mathcal{A}$ be a unital C*-algebra and $d\in \mathbb{N}$. We say that a set of vectors $\{u_1, u_2, \dots, u_n\}$ in $\mathcal{A}^d$ is in \textbf{modular radial isotropic position} with respect to a coefficient vector $c\coloneqq (c_1, c_1, \dots, c_d)^T \in \mathcal{A}^d$ if following conditions hold. 
\begin{enumerate}[\upshape(i)]
	\item $\langle u_j, u_j \rangle $ is invertible for all $1\leq j \leq n$. 
	\item \begin{align*}
		\sum_{j=1}^nc_j\langle u_j, u_j \rangle^\frac{-1}{2}u_j(\langle u_j, u_j \rangle^\frac{-1}{2}u_j)^*=I_d.
	\end{align*}
\end{enumerate}
\end{definition}
\begin{problem}
\textbf{Whether there is a characterization for a set of vectors in Hilbert C*-module which can be put into modular radial isotropic position with respect to a coefficient vectors?}	
\end{problem}
 Using chordal distance between subspaces it is proved in \cite{CAHILLCASAZZA} that Projection Problem implies Paulsen Problem. However, even though we can define the notion of trigonometric functions in C*-algebras (see \cite{MAHESH123}) it seems that we can not define the notion of chordal distance as defined in  \cite{CONWAYHARDIN}. However, motivated from characterization of chordal distance between  subspaces,  we propose the following definition of chordal distance. 
  \begin{definition}
  	\textbf{(Modular Chordal Distance) Let $\mathcal{A}$ be a C*-algebra with invariant basis number property. Let  $P, Q: \mathcal{A}^d \to \mathcal{A}^d$ be rank $m$ projections onto sub modules  $\mathcal{M}$ and $\mathcal{N}$ of $\mathcal{A}^d$, respectively. Let $[P]$ and $[Q]$ be matrices of $P$ and $Q$,  respectively, w.r.t. standard orthonormal  basis for  $\mathcal{A}^d$.  We define the modular chordal distance  between $\mathcal{M}$ and $\mathcal{N}$ as 
  		\begin{align*}
  			\operatorname{dist}_{\operatorname{ModChor}}(\mathcal{M}, \mathcal{N})\coloneqq \left\|m-\frac{1}{2}\left(\operatorname{Trace}([P][Q])+\operatorname{Trace}([Q][P])\right)\right\|^\frac{1}{2}.
  	\end{align*}}
  \end{definition}
            \begin{remark}
  	We believe strongly that modular Paulsen problem and modular projection problem are solvable at least for W*-algebras (von Neumann algebras) or C*-algebras with invariant basis number (IBN) property (we refer \cite{GIPSON} for IBN properties of C*-algebras).
  \end{remark}
We terminate by formulating following conjectures and a problem which we believe to have lot of importance in operator algebras. First two conjectures are based on Bourgain-Tzafriri Restricted Invertibility Theorem \cite{BOURGAINTZAFRIRI, BOURGAINTZAFRIRI2, TROPP, VERSHYNIN, CASAZZAPFANDER, SPIELMANSRIVASTAVA, CASAZZATREMAIN, SRIVASTAVA, NAORYOUSSEF, NAOR, MARCUSSPIELMANSRIVASTAVA, YOUSSEF} and the remaining is based on Johnson-Lindenstrauss Flattening Lemma \cite{FRANKLMAEHARA, JOHNSONLINDENSTRAUSS, FOUCARTRAUHUT, MATOUSEKBOOK, DASGUPTAGUPTA, ACHLIOPTAS, INDYKMOTWANI, AILONCHAZELLE2009, AILONCHAZELLE, VENKATASUBRAMANIANWANG, KRAHMERWARD, MATOUSEK, JACQUES, LARSENNELSON, LARSENNELSON2017, KANENELSON, KANEMEKANELSON, JAYRAMWOODRUFF, KLARTAGMENDELSON, FREKSEN, MENDEL, KNOLL, BURRGAOKNOLL, DASGUPTAKUMARSARLOS, DEDEZALAURENT, INDYK2000, INDYKNAOR, ALON, ARRIAGAVEMPALA, BECCHETTIBURYCOHENADDADGRANDONISCHWIEGELSHOHN, BARTALRECHTSCHULMAN, ALONKLARTAG, AILONLIBERTY, AILONLIBERTY2, VYBIRAL, HINRICHSVYBIRAL}.
\begin{conjecture}\label{RC} 
	\textbf{[(Commutative) Modular Bourgain-Tzafriri Restricted Invertibility Conjecture]
		Let $\mathcal{A}$ be a  unital commutative C*-algebra and $\mathcal{I}(\mathcal{A})$ be the set of all invertible elements of $\mathcal{A}$.   For $d\in \mathbb{N}$, let $ \mathbb{M}_{d\times d}(\mathcal{A})$ be the set of all $d$ by $d$ matrices over $\mathcal{A}$. For  $ M\in \mathbb{M}_{d\times d}(\mathcal{A})$, let $\det (M) $ be the determinant of $M$. Let $\mathcal{A}^d$ be the standard (left) Hilbert C*-module over 	$\mathcal{A}$ and    $\{e_j\}_{j=1}^d$ be the  canonical orthonormal basis for $\mathcal{A}^d$. There are universal real  constants $A>0$, $c>0$ ($A$ and $c$ may  depend upon C*-algebra $\mathcal{A}$) satisfying the following property. If  $d \in \mathbb{N}$ and  $ M \in \mathbb{M}_{d\times d}(\mathcal{A})$    with  $\langle Me_j, Me_j \rangle =1$, $\forall 1\leq j\leq d$ and $\det (M) \in \mathcal{I}(\mathcal{A})\cup\{0\}$, then there exists a subset $\sigma \subseteq \{1, \dots, d\}$ of cardinality 
		\begin{align*}
			\operatorname{Card}(\sigma)\geq \frac{cd}{\|M\|^2}
		\end{align*}
		such that 
		\begin{align*}
			\sum_{j \in \sigma}\sum_{k \in \sigma}a_j \langle Me_j, Me_k\rangle a_k^*=\left\langle \sum_{j \in \sigma}a_jMe_j, \sum_{k \in \sigma}a_kMe_k\right\rangle \geq A \sum_{j \in \sigma}a_ja_j^*, \quad \forall a_j \in  \mathcal {A}, \forall j \in \sigma, 
		\end{align*}	
		where $\|M\|$ is the norm of the Hilbert C*-module homomorphism defined by $M$ as $M:\mathcal{A}^d \ni x \mapsto Mx \in \mathcal{A}^d$.}
\end{conjecture}
To formulate Conjecture \ref{RC} for noncommutative unital C*-algebras  we use the notion of Manin matrices. We refer \cite{CHERVOVFALQUIRUBTSOV, CHERVOVFALQUI} for the basics of Manin matrices.
\begin{conjecture}\label{NRC}
	\textbf{[(Noncommutative) Modular Bourgain-Tzafriri Restricted Invertibility Conjecture]
		Let $\mathcal{A}$ be a  unital  C*-algebra and $\mathcal{I}(\mathcal{A})$ be the set of all invertible elements of $\mathcal{A}$.   For $d\in \mathbb{N}$, let $ \mathbb{MM}_{d\times d}(\mathcal{A})$ be the set of all $d$ by $d$ Manin matrices over $\mathcal{A}$. For  $ M\in \mathbb{MM}_{d\times d}(\mathcal{A})$, let $\det^\text{column} (M) $ be the Manin determinant of $M$ by column expansion. Let $\mathcal{A}^d$ be the standard (left) Hilbert C*-module over 	$\mathcal{A}$ and    $\{e_j\}_{j=1}^d$ be the  canonical orthonormal basis for $\mathcal{A}^d$. There are universal real  constants $A>0$, $c>0$ ($A$ and $c$ may  depend upon C*-algebra $\mathcal{A}$) satisfying the following property. If  $d \in \mathbb{N}$ and  $ M \in \mathbb{MM}_{d\times d}(\mathcal{A})$    with  $\langle Me_j, Me_j \rangle =1$, $\forall 1\leq j\leq d$ and $\det^\text{column} (M) \in \mathcal{I}(\mathcal{A})\cup\{0\}$, then there exists a subset $\sigma \subseteq \{1, \dots, d\}$ of cardinality 
		\begin{align*}
			\operatorname{Card}(\sigma)\geq \frac{cd}{\|M\|^2}
		\end{align*}
		such that 
		\begin{align*}
			\sum_{j \in \sigma}\sum_{k \in \sigma}a_j \langle Me_j, Me_k\rangle a_k^*=\left\langle \sum_{j \in \sigma}a_jMe_j, \sum_{k \in \sigma}a_kMe_k\right\rangle \geq A \sum_{j \in \sigma}a_ja_j^*, \quad \forall a_j \in  \mathcal {A}, \forall j \in \sigma, 
		\end{align*}	
		where $\|M\|$ is the norm of the Hilbert C*-module homomorphism defined by $M$ as $M:\mathcal{A}^d \ni x \mapsto Mx \in \mathcal{A}^d$.} 
\end{conjecture}
\begin{remark}
	\begin{enumerate}[\upshape (i)]
		\item \textbf{We can surely formulate Conjecture \ref{RC} by removing the condition $\det (M) \in \mathcal{I}(\mathcal{A})\cup\{0\}$ and Conjecture \ref{NRC} by removing the condition Manin matrices and $\det^\text{column} (M) \in \mathcal{I}(\mathcal{A})\cup\{0\}$. But we strongly  believe that Conjectures  \ref{RC} and \ref{NRC}  fail with this much of generality.}
		\item If Conjecture \ref{RC} holds but Conjecture \ref{NRC} fails,  then we  can try Conjecture \ref{NRC} for W*-algebras or C*-algebras with  IBN property.
	\end{enumerate}
\end{remark}
\begin{problem}\label{MODULARPROBLEM}
	\textbf{Let $\mathscr{A}$ be the set of all unital C*-algebras. What is the best function $\phi: \mathscr{A}\times (0,1)\times \mathbb{N} \to (0,\infty)$ satisfying the following. Let $\mathcal{A}$ be a unital C*-algebra. 	There is a universal constant $C>0$ (which may depend upon $\mathcal{A}$) satisfying the following. Let $0<\varepsilon<1$, $M,N \in \mathbb{N}$ and $\mathbf{x}_1, \mathbf{x}_2, \dots, \mathbf{x}_M\in \mathcal{A}^N$. For each natural number 
		\begin{align*}
			m>C\phi (\mathscr{A}, \varepsilon, M)
		\end{align*}
		there exists a matrix $M\in \mathbb{M}_{m\times N}(\mathcal{A})$ such that 
		\begin{align*}
			(1-\varepsilon)\langle \mathbf{x}_j-\mathbf{x}_k, \mathbf{x}_j-\mathbf{x}_k\rangle \leq \langle M(\mathbf{x}_j-\mathbf{x}_k), M(\mathbf{x}_j-\mathbf{x}_k)\rangle \leq (1-\varepsilon)\langle \mathbf{x}_j-\mathbf{x}_k, \mathbf{x}_j-\mathbf{x}_k\rangle, \quad \forall 1 \leq j, k\leq M.
	\end{align*}}
\end{problem}
A particular case of Problem \ref{MODULARPROBLEM} is the following  conjecture.
\begin{conjecture}\label{MJLC}
	\textbf{(Modular Johnson-Lindenstrauss Flattening Conjecture)
		Let $\mathcal{A}$ be a unital C*-algebra. 	There is a universal constant $C>0$ (which may depend upon $\mathcal{A}$) satisfying the following. Let $0<\varepsilon<1$, $M,N \in \mathbb{N}$ and $\mathbf{x}_1, \mathbf{x}_2, \dots, \mathbf{x}_M\in \mathcal{A}^N$. For each natural number 
		\begin{align*}
			m>\frac{C}{\varepsilon^2}\log M,
		\end{align*}
		there exists a matrix $M\in \mathbb{M}_{m\times N}(\mathcal{A})$ such that 
		\begin{align*}
			(1-\varepsilon)\langle \mathbf{x}_j-\mathbf{x}_k, \mathbf{x}_j-\mathbf{x}_k\rangle \leq \langle M(\mathbf{x}_j-\mathbf{x}_k), M(\mathbf{x}_j-\mathbf{x}_k)\rangle \leq (1-\varepsilon)\langle \mathbf{x}_j-\mathbf{x}_k, \mathbf{x}_j-\mathbf{x}_k\rangle, \quad \forall 1 \leq j, k\leq M.
	\end{align*}}	
\end{conjecture}  
\begin{remark}
	We believe that Conjecture \ref{MJLC}  holds  at least for W*-algebras (von Neumann algebras) or C*-algebras with IBN property. 	
\end{remark}

  \begin{remark}
  	\textbf{Modular Welch bounds} are derived in \cite{MAHESHKRISHNA} and a collection of problems for Hilbert C*-modules including \textbf{Modular Zauner conjecture} have  been formulated there.
  \end{remark}

 \bibliographystyle{plain}
 \bibliography{reference.bib}

\end{document}